%% file: main.tex
\documentclass{article}

\usepackage[a4paper,margin=2cm]{geometry}
\usepackage{amssymb,amsmath,amsthm,color,graphicx,dsfont,mathrsfs,stmaryrd,mathtools}
\usepackage{braket}
\usepackage{authblk}
\usepackage[colorlinks=true,allcolors=blue]{hyperref}
\usepackage[capitalise]{cleveref}

\usepackage[
    backend=biber,
    maxbibnames=6,
    maxcitenames=6,
    style=alphabetic,
    sorting=nyt,
    isbn=false,
    url=false,
    eprint=false,
giveninits=true,
date=year]{biblatex}
\newtheorem{theorem}{Theorem}
\newtheorem{lemma}{Lemma}

\newcommand{\C}{\mathbb{C}}
\newcommand{\R}{\mathbb{R}}

\newcommand{\xL}{\mathcal{L}}
\newcommand{\xH}{\mathcal{H}}

\newcommand{\xtr}[1]{\operatorname{Tr}\left( #1 \right)}
\newcommand{\xK}{\mathcal{K}}
\newcommand{\xN}{\mathbb{N}}
\newcommand{\xW}{\mathcal{W}}

\newcommand{\rN}{{\boldsymbol{\rho}_{(N)}}}

\newcommand{\norm}[1]{\| #1 \|}

\let\bf\boldsymbol

\newcommand{\Id}{\operatorname{Id}}

\newcommand{\orho}{{\bf \rho}}
\newcommand{\oM}{{\bf M}}
\newcommand{\oa}{{\bf a}}
\newcommand{\ob}{{\bf b}}

\newcommand{\oL}{{\bf L}}
\newcommand{\oH}{{\bf H}}

\newcommand{\oP}{{\bf P}}

\newcommand{\oA}{{\bf A}}
\newcommand{\oN}{{ \bf N}}

\newcommand{\opr}{{\bf r}}

\newcommand{\oLambda}{{\bf \Lambda}}
\newcommand{\osigma}{{\bf \sigma}}
\newcommand{\Ndissip}{N_j}

\title{Convergence Analysis of Galerkin Approximations for the Lindblad Master Equation}

\author[1]{Rémi Robin 	\thanks{remi.robin@minesparis.psl.eu}}
\author[1]{Pierre Rouchon 	\thanks{pierre.rouchon@minesparis.psl.eu}}
\affil[1]{Laboratoire de Physique de l'\'Ecole Normale Supérieure, Mines Paris, Inria, CNRS, ENS-PSL, Sorbonne Université, PSL Research University, Paris, France}
\begin{document}
\maketitle
\input{abstract.tex}
\tableofcontents
\input{intro.tex}
\input{cvg_galerkin.tex}
\input{example.tex}
\input{several_modes.tex}
\input{ccl.tex}
\appendix
\section{Notations}
\input{notation.tex}
\printbibliography
\end{document}

%% file: abstract.tex
This paper analyzes the numerical approximation of the Lindblad master equation on infinite-dimensional Hilbert spaces. We employ a classical Galerkin approach for spatial discretization and investigate the convergence of the discretized solution to the exact solution. Using \textit{a priori} estimates, we derive explicit convergence rates and demonstrate the effectiveness of our method through examples motivated by autonomous quantum error correction.

%% file: intro.tex
\section{Introduction}
\input{motivation.tex}

%% file: motivation.tex
\subsection{Motivations and Description of the Contributions}
Open quantum systems weakly coupled to their environment can be modeled by the Lindblad equation, also known as the Gorini--Kossakowski--Sudarshan--Lindblad equation~\cite{lindbladGeneratorsQuantumDynamical1976,goriniCompletelyPositiveDynamical1976,chruscinskiBriefHistoryGKLS2017}. The state of the system is described by a density operator $\orho$ on a complex separable Hilbert space $\xH$. A density operator is a positive semidefinite, self-adjoint, trace-class operator with trace $1$ on $\mathcal{H}$. The evolution of $\orho$ depends on the Hamiltonian $\oH$, which is a self-adjoint operator on $\xH$, and a set of operators $(\oL_j)$, not necessarily self-adjoint, called jump operators. From now on, we restrict ourselves to the case of a finite number of jump operators $(\oL_j)_{1\leq j\leq \Ndissip}$. The Lindblad equation then reads formally as 
\begin{align}
    \label{eq_Lindblad}
    \frac{d}{dt}\orho(t) =\xL(\orho(t)) \coloneq -i[\oH,\orho(t)] + \sum_{j=1}^{\Ndissip} D[\oL_j](\orho(t)),\quad \orho(0)=\orho_0,
\end{align}
with
\begin{align}
    \label{eq_dissipative}
    D[\oL](\orho)=\oL\orho \oL^\dag -\frac{1}{2}( \oL^\dag \oL \orho + \orho \oL^\dag \oL).
\end{align}
We aim to approximate the solution of \cref{eq_Lindblad} in the scenario where the Hilbert space $\mathcal{H}$ is infinite-dimensional and the operators $\oH$ and $(\oL_j)_{1 \leq j \leq \Ndissip}$ may be unbounded. A classical approach to this problem involves a Galerkin-type method: First, consider an increasing sequence of finite-dimensional Hilbert spaces $(\mathcal{H}_N)_{N \in \mathbb{N}}$ such that $\mathcal{H}_N \subset \mathcal{H}$ and $\mathcal{H}_N$ is contained within the domains $D(\oH)$ and $D(\oL_j)$. Let $\oP_N$ denote the orthogonal projector onto the subspace $\mathcal{H}_N$. Then, define the bounded operators
    \begin{align}
        \oH_N = \oP_N \oH \oP_N, \quad \oL_{j,N} = \oP_N \oL_j \oP_N.
        \label{eq_truncated}
    \end{align}
 and construct a new bounded Lindbladian operator
    \begin{align}
        \mathcal{L}_N(\orho) \coloneqq -i[\oH_N, \orho] + \sum_j D[\oL_{j,N}](\orho).
        \label{eq:Truncated_lindbladian}
    \end{align}
Note that $\oL_{j,N}^\dag \oL_{j,N}$ is usually different from $\oP_N \oL_j^\dag \oL_j \oP_N$, implying that the alternative choice of truncation $\oP_N D[\oL_{j}](\oP_N \orho \oP_N) \oP_N$ is not equivalent to our choice, and does not necessarily lead to a Lindbladian operator.
We can now define a numerical approximation $\rN$ of $\rho$ as the solution of the Lindblad equation with support on $\xH_N$
\begin{align}
    \frac{d}{dt}\rN=\xL_N(\rN),\quad \rN(0)=\oP_N \orho_0 \oP_N.
\end{align}
The solution of this equation is always well-defined and the semigroup is completely positive and trace preserving as $\xL_N$ is a bounded Lindbladian operator \cite{lindbladGeneratorsQuantumDynamical1976}.
To the best of the authors' knowledge, although several results exist for the Hamiltonian case ($\Ndissip=0$)--see, for example, recent advances in \cite{tongProvablyAccurateSimulation2022,fischerQuantumParticleWrong2025,pengQuantumSimulationBosonrelated2025}--there is no available result that guarantees the convergence of $\rho_N(t)$ towards $\rho(t)$ when $\Ndissip > 0$. Notably, recent work by the authors in \cite{etienneyPosterioriErrorEstimates2025} provides \textit{a posteriori} estimates on the trace-norm distance $\|\rho_N(t)-\rho(t)\|_1$ for a large class of operators $\oH$ and $(\oL_j)$, but does not include a proof of convergence. Additionally, in \cite{robinUnconditionallyStableTime2025}, the authors utilize \textit{a priori} estimates to analyze the time-discretization of some unbounded Lindblad equations for a class of Completely Positive Trace Preserving (CPTP) schemes. These \textit{a priori} estimates were developed in the 1990s to establish the well-posedness\footnote{More precisely, they establish the conservativity of the minimal semigroup, which ensures the uniqueness of the solution.} of the semigroup in \cite{chebotarevLindbladEquationUnbounded1997,chebotarevPrioriEstimatesQuantum2003,chebotarevSufficientConditionsConservativity1998}, and more recently, specifically for the setting of bosonic modes in \cite{gondolfEnergyPreservingEvolutions2024}.

In this article, we focus primarily on the latter setting, and more precisely on the case where the Hilbert space is given by $\xH = \ell^2(\mathbb{N})$ and the operators $\oH$ and $\oL_j$ are polynomials in the annihilation operator $\oa$ and creation operator $\oa^\dag$; we recall in Appendix \ref{subsec:notations} the definitions of these operators. Generalization to other settings requires the choice of a reference operator $\oLambda$ that plays the role of the number operator $\oN=\oa^\dag \oa$ both for defining regularity and approximation spaces, as well as technical domain assumptions. Some generalizations are briefly discussed in \cref{sec_generalization}.
Going back to the case of a single bosonic mode, we denote by $\xW^{0,1}$ the Banach space of trace-class operators and by $\| \cdot\|_1$ the trace norm. Following \cite{gondolfEnergyPreservingEvolutions2024}, we define the so-called Sobolev bosonic spaces for $k \in \mathbb{R}_+$,
    \begin{align}
        \label{def_Hk}
        \xH^k &\coloneqq D\big((\oa^\dag \oa)^{k/2}\big), \\
        \label{def_Wk}
        \xW^{k,1} &\coloneqq \{\,\orho\in\xW^{0,1} \;|\; (\oa^\dag \oa+\Id)^{k/2}\,\orho\,(\oa^\dag \oa+\Id)^{k/2} \in \xW^{0,1} \,\}.
    \end{align}
Note that we use a slightly different convention for the indexing of $\xW^{k,1}$, which corresponds to $\xW^{k/2,1}$ in \cite{gondolfEnergyPreservingEvolutions2024}.
We equip the space \(\xW^{k,1}\) with the norm
\[
    \|\orho\|_{\xW^{k,1}} \coloneqq \big\| (\oa^\dag \oa+\Id)^{k/2}\,\orho\,(\oa^\dag \oa+\Id)^{k/2} \big\|_1,
\]
which ensures that \(\xW^{k,1}\) is a Banach space. 
Equivalently, \(\orho\in\xW^{k,1}\) if and only if there exists \(\osigma\in\xW^{0,1}\) such that
\[
    \orho = (\oa^\dag \oa+\Id)^{-k/2}\,\osigma\,(\oa^\dag \oa+\Id)^{-k/2},
\]
and in that case we set \(\|\orho\|_{\xW^{k,1}} := \|\osigma\|_1\). The map
\(\osigma \mapsto (\oa^\dag \oa+\Id)^{-k/2}\,\osigma\,(\oa^\dag \oa+\Id)^{-k/2}\) is an isometric isomorphism between \(\xW^{0,1}\) and \(\xW^{k,1}\). Next, let $\oP_N$ denote the spectral projector onto $[0,N]$ of $\oa^\dag \oa + \Id$, and consider the Galerkin approximation spaces as $\xH_N=\oP_N\xH$. Namely, $\xH_N=\text{span}\{ \ket{k} \mid k \leq N-1\}$, commonly known in the physics community as the truncated Fock basis. The main result of this paper is to leverage \textit{a priori} estimates that provide regularity related to the Sobolev spaces $\xW^{k,1}$ to obtain convergence rates of $\rN$ towards $\orho$.

The rest of the paper is organized as follows: In \cref{sec:one_mode}, we study the single bosonic mode case introduced above. First, we recall how to obtain regularity results from \textit{a priori} estimates in \cref{subsec:apriori_one_mode}, then in \cref{subsec_cvg_1mode} we establish our main result, \cref{th_main_1_mode}, which ensures convergence of $\rN$ for the setting described above. Then, two examples are provided in \cref{ex_1mode}. In \cref{sec_generalization}, we discuss how to extend our results to more general settings and provide an example with two bosonic modes in \cref{sub_sec:cat_buffer}. We conclude and discuss future work in \cref{sec:conclusion}. Notations are collected in Appendix \ref{subsec:notations} for the reader's convenience.

%% file: cvg_galerkin.tex
\section{Convergence for a Single Bosonic Mode}
\label{sec:one_mode}

\subsection{A Priori Estimates}
\label{subsec:apriori_one_mode}

For the Lindblad master equation, regularity estimates typically assert the following: given a self-adjoint operator $\oLambda \geq 0$, we assume there exists a constant $C > 0$ such that, for sufficiently regular, see \cref{subsec:general_tools}, the estimate
\begin{align}
    \xtr{\xL(\orho )\oLambda}\leq C \xtr{\orho  \oLambda},
\end{align}
holds. Then, under additional domain assumptions, one may obtain
\[
\xtr{\orho _t \oLambda}\leq e^{Ct} \xtr{\orho_0 \oLambda},
\]
where $\orho_t = e^{t\xL}\orho_0$ is the solution\footnote{Uniqueness of the solution is given under a condition known as conservativity of the minimal semigroup which is often obtained precisely with a priori estimates of this form, see e.g. \cite{chebotarevSufficientConditionsConservativity1993,chebotarevSufficientConditionsConservativity1998}. For our case, it is a consequence of \cref{th_apriori_single_mode}.} of the Lindblad equation at time $t$ with initial condition $\orho_0$. 
The result below is of this form with $\oLambda = (\Id + \oa^\dag \oa)^k$.

\begin{theorem}[Theorem 3.1 of \cite{gondolfEnergyPreservingEvolutions2024}]
\label{th_apriori_single_mode}
Let $\oH$ and $(\oL_j)$ be polynomials in the creation and annihilation operators, and assume $\oH$ is self-adjoint. Denote their degrees by $p_H$ and $p_j$, respectively.

Assume further that there exists an increasing sequence $(k_r)_{r\in \xN}\in (\mathbb{R}_+)^\mathbb{N}$ with $k_r \to \infty$ and constants $w_{k_r} \geq 0$ such that, for all positive semidefinite $\orho \in \xW^f$\footnote{$\xW^f$ is the set of finite-rank operators on $\xH$ whose range is included in $\xH^f=\text{span}\{\ket{n}, n \in \mathbb{N}\}$, see Appendix \ref{subsec:notations} for details}, and all $r \in \xN$,
\begin{align}
    \label{eq:apriori_estimate}
    \xtr{\xL(\orho)(\oa^\dag \oa+ \Id)^{k_r}} \leq w_{k_r} \xtr{\orho (\oa^\dag \oa+ \Id)^{k_r}}.
\end{align}
Then, the closure of $(\xL, \xW^f)$ generates a strongly continuous, positivity-preserving semigroup on $\xW^{k,1}$ for all $k \in \mathbb{R}_+$, with the estimate
\begin{align}
    \|e^{t\xL}\|_{\xW^{k,1}\to \xW^{k,1}} \leq e^{w_k t }, \quad \forall t \geq 0,
\end{align}
where $w_k = \frac{k_{r_1}-k}{k_{r_1}-k_{r_0}}w_{k_{r_0}} + \frac{k-k_{r_0}}{k_{r_1}-k_{r_0}}w_{k_{r_1}}$ for $r_0, r_1$ such that $k_{r_0} \leq k < k_{r_1}$. Moreover, for $k = 0$, the semigroup is contractive and trace-preserving.
\end{theorem}

While the previous estimate increases exponentially with time, it is quite common to obtain a stronger version of \cref{eq:apriori_estimate}. Namely, assume that there exist $\mu_{k_r},\eta_{k_r}\geq 0$ such that for any $\orho \in \xW^f$, $\orho \geq 0$, and $\xtr{\orho}=1$,
\begin{align}
    \label{eq:apriori_estimate_}
    \xtr{\xL(\orho)(\oa^\dag \oa+ \Id)^{k_r}} \leq \mu_{k_r} -\eta_{k_r} \xtr{\orho (\oa^\dag \oa+ \Id)^{k_r}},
\end{align}

Then, we can easily obtain (see, e.g., \cite[Prop. 4.1]{gondolfEnergyPreservingEvolutions2024}) that for any $r$ and $\orho_0 \in \xW^{k_r,1}$, the following uniform-in-time estimate holds:
\begin{align}
    \| e^{t\xL}\orho_0 \|_{\xW^{k_r,1}} \leq \max( \|\orho_0\|_{\xW^{k_r,1}}, \mu_{k_r}/\eta_{k_r}), \quad \forall t \geq 0.
\end{align}

\subsection{Convergence of the Galerkin Approximations}
\label{subsec_cvg_1mode}

We now state the main result of this section.

\begin{theorem}
\label{th_main_1_mode}
Assume the hypotheses of \cref{th_apriori_single_mode} hold. Let $d = \max(p_H, 2p_j)$ and fix $k > d$. Then, there exists a constant $C_k \geq 0$ such that for every initial condition $\orho_0 \in \xW^{k,1}$, we have
\begin{align}
    \label{eq_th_main_1_mode}
    \|\orho(t)-\rN(t)\|_1 \leq \frac{C_k t }{N^{(k-d)/2}} \| \orho\|_{L^\infty(0,t; \xW^{k,1})}, \quad \forall t \geq 0.
\end{align}
\end{theorem}

The remainder of this section is dedicated to the proof of \cref{th_main_1_mode}.

We begin by introducing several key tools in \cref{lem_polynomial_degree_bounded,lem_bound,lem_Duhamel}. Next, \cref{lem:reg_vs_rate} is used to trade regularity for control over the localization of $\orho$ on low Fock states. Finally, the core estimate needed to complete the proof is provided in \cref{lem:key_estimate}.

We first recall a lemma on polynomials of creation and annihilation operators.

\begin{lemma}[{\cite[Lemma B.1]{gondolfEnergyPreservingEvolutions2024}}]
\label{lem_polynomial_degree_bounded}
Let $P \in \C[X,Y]$ be a polynomial of degree $d$. Then $(P(\oa,\oa^\dag), \xH^f)$ is closable, and there exists $c \geq 0$ such that
\begin{align}
    \|P(\oa,\oa^\dag) \ket{\psi}\| \leq c \|(\Id + \oa^\dag \oa)^{d/2} \ket{\psi}\|, \quad \forall \ket{\psi} \in \xH^f.
\end{align}
\end{lemma}

This lemma has two natural consequences. First, the operator $P(\oa,\oa^\dag)$ can be extended to a bounded operator from $\xH^d$ to $\xH$. Second, for any $k \in \xN$, the operator $(\Id + \oa^\dag \oa)^k P(\oa,\oa^\dag)$ is bounded on $\xH^{2k+d}$, meaning that $P(\oa,\oa^\dag)$ is bounded from $\xH^{2k+d}$ to $\xH^{2k}$. Classical interpolation, see e.g. \cite[Chap. 4]{berghInterpolationSpacesIntroduction1976}, then implies that $P(\oa,\oa^\dag)$ is also bounded from $\xH^{k+d}$ to $\xH^k$ for any $k \in \R_+$.

We next present a simple estimate that is crucial for bounding terms in the Lindbladian.

\begin{lemma}
\label{lem_bound}
Let $s \geq 0$ and let $\oM_0$, $\oM_1$ be bounded operators from $\xH^s$ to $\xH$. Then, for any $\orho \in \xW^{s,1}$, the operator $\oM_0 \orho \oM_1^\dag$ is trace-class, and
\begin{align}
    \|\oM_0 \orho \oM_1^\dag\|_1 \leq \|\oM_0\|_{\xH^s \to \xH} \|\orho\|_{\xW^{s,1}} \| \oM_1\|_{\xH^s \to \xH}.
\end{align}
Here, $\oM_1^\dag$ denotes the adjoint of $\oM_1$ viewed as a map $\xH^s \to \xH$, and is thus bounded from $(\xH)^*$ to $(\xH^s)^* \eqcolon \xH^{-s}$.
\end{lemma}
\begin{proof}
Since $\orho \in \xW^{s,1}$, we can write $\orho = (\Id + \oa^\dag \oa)^{-s/2} \osigma (\Id + \oa^\dag \oa)^{-s/2}$ for some $\osigma \in \xW^{0,1}$. Then,
\begin{align*}
    \|\oM_0 (\Id + \oa^\dag \oa)^{-s/2} \osigma (\Id + \oa^\dag \oa)^{-s/2} \oM_1^\dag\|_1 
    &\leq \|\oM_0 (\Id + \oa^\dag \oa)^{-s/2}\|_\infty \|\osigma\|_1 \|(\Id + \oa^\dag \oa)^{-s/2} \oM_1^\dag\|_\infty \\
    &= \|\oM_0\|_{\xH^s \to \xH} \|\orho\|_{\xW^{s,1}} \| \oM_1\|_{\xH^s \to \xH}.
\end{align*}
\end{proof}

\begin{lemma}[Duhamel formula]
    \label{lem_Duhamel}
    Assume that the hypotheses of \cref{th_main_1_mode} hold and that $\orho_0 \in \xW^{k,1}$, then
    \begin{align}
        \label{eq:lem_Duhamel}
        \|\orho(t)-\orho_N(t)\|_1 \leq   \|\orho_0-\orho_N(0)\|_1 + \int_0^t \| (\xL-\xL_N)(\orho(s))\|_1 ds
    \end{align}
\end{lemma}
\begin{proof}
    Let $\opr(t)=\orho(t)-\rN(t)$. By \cref{th_apriori_single_mode}, $(e^{t\xL})_{t\geq 0}$ is a strongly continuous semigroup on $\xW^{k,1}$; thus for any $t\geq 0$, $\orho(t)\in \xW^{k,1}$. Since $\rN(t)$ stays in the finite-dimensional approximation space, it clearly belongs to $\xW^{k,1}$. Using \cref{lem_bound,lem_polynomial_degree_bounded}, we obtain a constant $C>0$ such that for any $\tilde \orho \in \xW^{f}$,
    \begin{align}
        \|\xL(\tilde \orho)\|_1 \leq C \|\tilde \orho\|_{\xW^{\max(p_H,2p_j),1}}.
    \end{align}
    As a consequence, $\xW^{k,1}$ is included in the domain of the generator of the semigroup, that is, using \cref{th_apriori_single_mode}, the closure of $(\xL,\xW^f)$ for the $\xW^{0,1}$ norm.

    On the other hand, $\xL_N$ is a finite-rank operator and is bounded on any $\xH^s$ for $s\in \mathbb{R}$. As a consequence, for any positive time $T \geq 0$, we obtain $\opr\in \mathcal{C}^1([0,T],D(\xL))$ and
    \begin{align}
        \frac{d}{dt}\opr(t)= \xL(\orho(t))-\xL_N(\rN(t))=\xL_N(\opr(t))+(\xL-\xL_N)(\orho(t)).
    \end{align}
    Using that $\xL_N$ is a bounded operator, we easily get the Duhamel formula,
    \begin{align}
        \opr(t)=e^{t\xL_N}\opr(0)+\int_0^t e^{(t-s)\xL_N}(\xL-\xL_N)(\orho(s))ds.
    \end{align}
    Next, since $(e^{t\xL_N})_{t\geq 0}$ is a completely positive trace-preserving (CPTP) map, it contracts the trace norm. Together with the triangle inequality, this concludes the proof.
\end{proof}

The next lemma allows us to trade regularity for control of the tail of $\orho$.
\begin{lemma}
    \label{lem:reg_vs_rate}
    Let $0 \leq s_1 \leq s_2$ and consider $\oP_N$ the spectral projector on $[0,N]$ of $\Id + \oa^\dag \oa$. The following inequality holds
    \begin{align}
        \|\oP_N^\perp\|_{\xH^{s_2}\to \xH^{s_1}}\leq \frac{1}{N^{(s_2-s_1)/2}}.
    \end{align}
\end{lemma}
\begin{proof}
    Let us denote by $\oLambda$ the positive self-adjoint operator $\Id + \oa^\dag \oa$ with domain $\xH^{2}$. Using the spectral decomposition of $\oLambda$, we have for any $s \geq 0$,
    \begin{align}
    \oLambda ^s&=\oLambda^s \oP_N + \oLambda^s \oP_N^\perp\\
    & \geq \oLambda^s \oP_N + N^s \oP_N^\perp
    \end{align}
    A useful consequence is the following Markov-type operator inequality $\oLambda ^s \geq N^s \oP_N^\perp$, which allows us to obtain
    \begin{align}
        \label{eq_markov_sobolev}
        \oLambda^{s_1/2}\oP_N^\perp \oLambda^{-s_2/2}\leq \frac{\oLambda^{(s_2-s_1)/2}}{N^{(s_2-s_1)/2}} \oLambda^{(s_1-s_2)/2}=\frac{\Id}{N^{(s_2-s_1)/2}},
    \end{align}
    where we used that $\oLambda$ commutes with its spectral projector $\oP_N$. Identifying the left-hand side of \cref{eq_markov_sobolev} as $\|\oP_N^\perp\|_{\xH^{s_2}\to \xH^{s_1}}$ concludes the proof.
\end{proof}

To finish the proof of \cref{th_main_1_mode}, we have to bound the two terms of the right-hand side of \cref{eq:lem_Duhamel}. The first one is simple:
\begin{align}
    \| \orho_0-\oP_N \orho_0 \oP_N\|_1 &= \| \oP_N^\perp \orho_0 \oP_N + \orho_0 \oP_N^\perp\|_1\\
    &\leq \| \oP_N^\perp \orho_0\|_1 \|\oP_N\|_\infty + \|\orho_0 \oP_N^\perp\|_1\\
    &=2 \| \oP_N^\perp \orho_0\|_1.
\end{align}
From \cref{lem:reg_vs_rate}, $\| \oP_N^\perp \orho_0\|_1\leq  \frac{1}{N^{k/2}}\|\orho_0\|_{\xW^{k,1}}$, which is smaller than the right-hand side of \cref{eq_th_main_1_mode}. Hence, it remains to bound the second term of the right-hand side of \cref{eq:lem_Duhamel}, which is the object of the next lemma.
\begin{lemma}
    \label{lem:key_estimate}
    There exists $C_k>0$ such that for all $\orho\in \xW^{k,1}$,
    $\| (\xL-\xL_N)(\orho)\|_1\leq \frac{C_k}{N^{(k-d)/2}} \|\orho\|_{\xW^{k,1}}$
\end{lemma}

\begin{proof}
    We establish the inequality for $\orho\in \xW^f$, a simple density argument then allows us to extend to $\xW^{k,1}$. Let us start with the Hamiltonian part of the Lindbladian:
    \begin{align}
        \|[\oH-\oH_N,\orho]\|_1 &\leq 2 \| (\oH-\oH_N)\orho\|_1\\
        & \leq 2 (\| \oP_N \oH \oP_N^\perp \orho \|_1 + \| \oP_N^\perp \oH \orho \|_1)\\
        & \leq 2 (\| \oH \oP_N^\perp \orho \|_1 + \| \oP_N^\perp \oH \orho \|_1).
    \end{align}
    Then, using \cref{lem:reg_vs_rate}, we get
    \begin{align}
        \| \oH \oP_N^\perp \orho \|_1+ \| \oP_N^\perp \oH \orho \|_1&
        \leq \left( \|\oH\|_{\xH^{p_H}\to \xH} \|\oP_N^\perp\|_{\xH^k \to \xH^{p_H}} + \|\oP_N^\perp\|_{\xH^{k-p_H}\to \xH} \|\oH\|_{\xH^{k}\to \xH^{k-p_H}} \right)\|\orho\|_{\xW^{k,1}}\\
        &\leq \frac{\|\oH\|_{\xH^{p_H}\to \xH}+\|\oH\|_{\xH^{k}\to \xH^{k-p_H}}}{N^{(k-p_H)/2}} \|\orho\|_{\xW^{k,1}}.
    \end{align}

    Let us now address the dissipative part, for a jump operator $\oL$, we have
    \begin{align}
        \label{eq_dissipator_main}
        \|(D[\oL]-D[\oL_N])(\orho)\|_1 \leq \|\oL \orho \oL^\dag - \oL_N\orho \oL_N^\dag\|_1 + 2\|(\oL^\dag \oL-\oL_N^\dag \oL_N)\orho\|_1
    \end{align}
    \paragraph{First term of the right-hand side of \cref{eq_dissipator_main}.} We split again:
    \begin{align}
        \label{eq_second_decompo}
        \|\oL \orho \oL^\dag -\oL_N\orho \oL_N^\dag\|_1\leq \|(\oL-\oL_N)\orho \oL^\dag\|_1 +\|\oL_N\orho (\oL^\dag-\oL_N^\dag)\|_1
    \end{align}
    The first term is handled using $ \oL -\oL_N= \oP_N^\perp \oL + \oP_N \oL \oP_N^\perp$, namely:
    \begin{align}
        \|(\oL-\oL_N)\orho \oL^\dag\|_1 & \leq \|\oP_N^\perp \oL \orho \oL^\dag\|_1+ \|\oP_N \oL \oP_N^\perp \orho \oL^\dag\|_1\\
        & \leq \frac{1}{N^{(k-p_j)/2}}\left( \|\oL \|_{\xH^{p_j}\to \xH}+\|\oL \|_{\xH^{k}\to \xH^{k-p_j}} \right) \|\orho\|_{\xW^{k,1}}\|\oL \|_{\xH^{p_j}\to \xH}.
    \end{align}
    Similarly, for the second term of \cref{eq_second_decompo}, we have
    \begin{align}
        \|\oL_N\orho (\oL^\dag- \oL_N^\dag)\|_1 &\leq \frac{1}{N^{(k-p_j)/2}}\|\oL_N\|_{\xH^{p_j}\to \xH} \|\orho\|_{\xW^{k,1}}\left(\|\oL \|_{\xH^{p_j}\to \xH}+\|\oL \|_{\xH^{k-p_j}\to \xH^{p_j}}\right) \\
        &\leq \frac{1}{N^{(k-p_j)/2}}\|\oL \|_{\xH^{p_j}\to \xH} \|\orho\|_{\xW^{k,1}} \left(\|\oL \|_{\xH^{p_j}\to \xH}+\|\oL \|_{\xH^{k-p_j}\to \xH^{p_j}}\right).
    \end{align}
    \paragraph{Second term of the right-hand side of \cref{eq_dissipator_main}.}
    We decompose the left factor as
    \begin{align}
        \label{eq_remain_second_part}
        (\oL^\dag \oL - \oP_N \oL^\dag \oP_N \oL \oP_N)\orho= (\oP_N^\perp \oL^\dag \oL + \oP_N \oL^\dag \oP_N^\perp \oL+ \oP_N \oL^\dag \oP_N \oL \oP_N^\perp)\orho
    \end{align}
    For the first term of the sum, we get
    \begin{align}
        \|\oP_N^\perp \oL^\dag \oL \orho \|_1&\leq \|\oP_N^\perp\|_{\xH^{k-2p_j}\to \xH} \| \oL^\dag \oL\|_{\xH^{k}\to \xH^{k-2p_j}} \|\orho\|_{\xW^{k,1}}\\
        &\leq \frac{1}{N^{(k-2p_j)/2}}\| \oL^\dag \oL\|_{\xH^{k}\to \xH^{k-2p_j}} \|\orho\|_{\xW^{k,1}},
    \end{align}
    The remaining term of \cref{eq_remain_second_part} are treated the same way.
    Hence, we have shown that
    \begin{align}
        \| (\xL-\xL_N)(\orho)\|_1\leq \frac{C_k}{N^{(k-\max (p_H,2p_j))/2}} \|\orho\|_{\xW^{k,1}}
    \end{align}
\end{proof}

%% file: example.tex
\subsection{Examples}
\label{ex_1mode}
\subsubsection{Example 1: Quantum Ornstein Uhlenbeck}
Consider the quantum Ornstein--Uhlenbeck (qOU) generator:
\begin{align}
    \xL(\orho) = \lambda^2 D[\oa](\orho) + \mu^2 D[\oa^\dag](\orho),
\end{align}
with $\lambda, \mu > 0$. This generator is widely used to model the interaction of a quantum harmonic oscillator with a thermal bath at non-zero temperature. Well-posedness and spectral properties are studied by mathematicians in \cite{ciprianiSpectralAnalysisFeller2000,carboneHypercontractivityQuantumOrnstein2008}, and \textit{a priori} estimates are provided in \cite{gondolfEnergyPreservingEvolutions2024}. We have the following result:
\begin{lemma}[{\cite[Lemma 4.2]{gondolfEnergyPreservingEvolutions2024}}]
    For any $k \in \mathbb{N}$, there exists an explicit constant $\mu_k$ such that for every initial condition $\orho_0 \in \xW^{k,1}$,
    \begin{align}
        \| e^{t\xL} \orho_0 \|_{\xW^{k,1}} \leq \begin{cases}
            \max\left( \| \orho_0 \|_{\xW^{k,1}},  \frac{2 \mu_k}{k(\lambda^2-\mu^2 )} \right) & \text{if $\lambda > \mu$},\\
            e^{t\frac{k}{8}(4 \mu^2+k)}\| \orho_0 \|_{\xW^{k,1}} & \text{if $\lambda \leq \mu$}.\\
        \end{cases}
    \end{align}
\end{lemma}
As a consequence, if $\orho_0 \in \xW^{k,1}$ for some $k>2$, \cref{th_main_1_mode} ensures the convergence at a rate $1/N^{(k-2)/2}$ of the Galerkin approximation $\rN$ towards $\orho$ in the trace norm, and the estimate grows linearly with time when $\lambda > \mu$ and exponentially otherwise.
\subsubsection{Example 2: Dissipative Cat-Qubit}
A more complex example for the bosonic code community is the dissipative cat qubit \cite{mirrahimiDynamicallyProtectedCatqubits2014}. After adiabatic elimination of the fast-decaying mode, the Lindblad equation takes the form
\begin{align}
    \xL(\orho) = \kappa_2 D[\oa^2 - \alpha^2](\orho).
\end{align}
Well-posedness and proof of convergence of this equation towards the \textit{code space} spanned by the two coherent states $\ket{\pm \alpha}$ is shown in \cite{azouitWellposednessConvergenceLindblad2016}. \textit{A priori} estimates on $\xW^{k,1}$ that are uniform in time are again provided in \cite[Section 4.2]{gondolfEnergyPreservingEvolutions2024}. Note that adding a quadratic Hamiltonian, a single photon loss term $\kappa_1 D[\oa](\orho)$ and/or a creation term $\kappa_1' D[\oa^\dag](\orho)$ do not change the existence of the uniform-in-time \textit{a priori} estimates. As before, if $\orho_0 \in \xW^{k,1}$ for some $k>4$, \cref{th_main_1_mode} ensures the convergence at a rate $1/N^{(k-4)/2}$ of the Galerkin approximation $\rN$ towards $\orho$ in trace norm with an estimate that grows linearly with time.

%% file: several_modes.tex
\section{Generalizations}
\label{sec_generalization}
\subsection{Main Tools to Tackle the General Case}
\label{subsec:general_tools}
If we summarize the proof of \cref{th_main_1_mode}, the key ingredient is the existence of an interesting self-adjoint operator $\oLambda\geq \Id$. Equipped with this operator, we can define the Sobolev-like spaces as done in \cref{def_Hk,def_Wk} by replacing $\oa^\dag \oa$ with $\oLambda$. More precisely, we need the following:
\begin{itemize}
    \item The Galerkin approximation spaces $\xH_N$ under consideration are defined as the spectral projector of $\oLambda$ on $[0,N]$. For practical purposes, $\oLambda$ should have a compact resolvent so that these spaces are finite-dimensional.
    \item There exist \textit{a priori} estimates formally of the form
    \begin{align}
        \label{eq:apriori_ccl}
        \xtr{\xL(\orho)\oLambda^k} \leq w_k \xtr{\orho \oLambda^k}
    \end{align}
    for sufficiently many density operators $\orho$ and $k$ large enough, or even better for the uniform-in-time case
        \begin{align}
            \label{eq:apriori_ccl_strong}
        \xtr{\xL(\orho)\oLambda^k} \leq \eta_k -\mu_k \xtr{\orho \oLambda^k}.
    \end{align}
    The goal is then to obtain estimates on the semigroup $e^{t\xL}$ for our new Sobolev-like spaces $\xW^{k,1}$. Obviously, \textit{sufficiently many density operators $\orho$} is a vague statement, and the technical details to make this rigorous depend on the setting under consideration. Introducing $\xL^*$ the adjoint of $\xL$ for the trace duality,
    \cref{eq:apriori_ccl} can be rewritten as
    \begin{align}
        \xL^*(\oLambda^k) \leq w_k \oLambda^k.
    \end{align}
    To give a precise meaning to this unbounded quadratic form inequality that allows us to obtain \textit{a priori} estimates, we refer, for example, to \cite{chebotarevPrioriEstimatesQuantum2003} or \cite[Section 3.6]{fagnolaQuantumMarkovSemigroups1999}.
    \item The operators $\oH$, $\oL_j$, and $\oL_j^\dag$ have to be controlled by $\oLambda$, in the sense that there exist $p_H,(p_j)>0$ such that these operators are bounded from $\xH^{k+p_H}$ to $\xH^k$ for all $k \geq 0$.
\end{itemize}

\subsection{Example 3: Dissipative Cat-Qubit with Buffer Cavity}
\label{sub_sec:cat_buffer}
A more complex example is the dissipative cat qubit with buffer cavity. This system corresponds to the physical implementation of a physical cat qubit when the adiabatic elimination of the buffer has not been performed and is an important example of quantum reservoir engineering. We consider two bosonic modes with annihilation operators $\oa$ and $\ob$. The Lindblad equation then reads 
\begin{align}
    \xL(\orho)= -i[(\oa^2-\alpha^2)\ob^\dag + (\oa^{\dag 2}-\alpha^2)\ob,\orho]+ \kappa_b D[\ob](\orho)
\end{align}
In \cite{robinConvergenceBipartiteOpen2024}, the authors prove well-posedness of the time evolution equation as well as the convergence of any trajectory towards the invariant manifold. Using the reference operator $\oLambda= (\frac{\oa^\dag \oa}{2}+ \ob^\dag\ob)$, they provide uniform-in-time \textit{a priori} estimates in $\xW^{k,1}$ for any $k \geq 0$. Note that obtaining a uniform-in-time approximation in this case does not directly follow from a version of \cref{eq:apriori_ccl_strong}. Then, it is easy to check that the operator $\oH= (\oa^2-\alpha^2)\ob^\dag + (\oa^{\dag 2}-\alpha^2)\ob$ (resp., $\oL=\sqrt{\kappa_b}\ob$ and its adjoint) is bounded from $\xH^{k+3}$ (resp., $\xH^{k+1}$) to $\xH^k$ for all $k$.

As a consequence, for every $k > 3$ and for every initial condition $\orho_0 \in \xW^{k,1}$, i.e., $\xtr{\orho_0 (\frac{\oa^\dag \oa}{2}+ \ob^\dag\ob)^k}<\infty$, the Galerkin approximation $\rN(t)$ defined with the truncated Hilbert space $\xH_N=\text{span}(\ket{n_a,n_b}: n_a/2+n_b \leq N)$ converges towards $\orho(t)$ in trace norm at a rate $1/N^{(k-3)/2}$ with an estimate that grows linearly with time.

%% file: ccl.tex
\section{Conclusion}
\label{sec:conclusion}

Our main contribution is \cref{th_main_1_mode}, which establishes that under suitable regularity assumptions on the initial condition $\orho_0 \in \xW^{k,1}$ with $k > d = \max(p_H, 2p_j)$, the Galerkin approximation $\rN(t)$ converges to the exact solution $\orho(t)$ in trace norm at an algebraic rate $N^{-(k-d)/2}$. This rate depends directly on the regularity of the initial condition. The flexibility of our framework is demonstrated through the generalization to multi-mode systems, where the key ingredients remain the same: a suitable reference operator $\oLambda$, corresponding to \textit{a priori} estimates, and appropriate domain assumptions on the generators.

Several interesting questions emerge from our work that merit further investigation:

\paragraph{Exponential convergence rates.} In the examples we studied, the \textit{a priori} estimates hold for all regularity levels $k$. This raises the question of whether exponential convergence in $N$ might be achievable under additional assumptions.

\paragraph{Regularization.} An interesting question is related to the regularization properties of the Lindblad semigroup. For example, given an initial condition without any prior regularity $\orho_0$ and time $t > 0$, under what conditions can we guarantee that $\orho(t) \in \xW^{k,1}$ for a given $k > 0$? For example, the authors believe that the Lindbladian given by $D[\oa]$ does not improve regularity, while $D[\oa^2]$ does. 

\paragraph{Time discretization.} Our analysis focuses solely on spatial discretization via Galerkin methods. A natural extension would combine our results with the time discretization analysis infinite dimension developed in \cite{robinUnconditionallyStableTime2025} to provide complete error estimates for fully discrete schemes.

\paragraph{Acknowledgments}
The authors thank Daniel Burgarth, Claude Le Bris, Mazyar Mirrahimi, Alain Sarlette, and Lev-Arcady Sellem for valuable discussions and insights that contributed to this work.

This project has received funding from the European Research Council (ERC) under the European Union’s Horizon
2020 research and innovation program (grant agreement No. 884762) and Plan France 2030 through the project ANR-22-PETQ-0006.

%% file: notation.tex
\label{subsec:notations}
We collect here the notations and definitions used throughout the paper. Note also that we set $\hbar=1$ and work with dimensionless quantities.

\begin{itemize}
    \item $\xH$ is a complex separable Hilbert space. Scalar products are denoted using Dirac's bra-ket notation, namely $\ket{x}$ is an element of $\xH$, whereas $\bra{x}$ is the linear form canonically associated to the vector $\ket{x}$.
    \item Operators on $\xH$ are denoted with bold characters such as $\oa, \ob, \orho$, $\oH$, $\oL$.
    \item $B(\xH)$ denotes the (von Neumann) algebra of bounded operators on $\xH$.
    $\Id$ or $\Id_\xH$ denotes the identity operator and $\norm{ \cdot }_\infty$ is the operator norm induced by the Hilbert norm on $\xH$.
    
    \item For an operator $\oA\in B(\xH)$, we say $\oA$ is positive (semidefinite) if it is self-adjoint and for all $u \in \xH$, $\bra{u}\oA \ket{u}> 0$ ($\geq 0$), also written as $\oA>0$ ($\oA \geq 0$). 
    \item If $\xH,\xH'$ are Hilbert spaces, we denote by $B(\xH,\xH')$ the Banach space of linear applications from $\xH$ to $\xH'$ that are continuous for the operator norm. We denote the operator norm of $B(\xH,\xH')$ by $\|\oM\|_{\xH \to \xH'}$, and as mentioned above, $\|\oM\|_{\xH \to \xH}$ is usually denoted $\|\oM\|_\infty$.
    \item The Banach space of trace-class operators on $\xH$ is denoted $\xW^{0,1}$ and the trace norm of $\orho\in \xW^{0,1}$ is $\|\orho\|_1$. 
    \item For any Banach space $\xK$ and $T>0$, the Banach space of essentially bounded functions from $[0,T]$ to $\xK$ is denoted $L^\infty(0, T;\xK)$.
    
    \item The Hilbert space $\ell^2(\xN,\C)$ is often referred to as a single bosonic mode. We denote its canonical basis, often called Fock basis, by $(\ket{n})_{n\in \xN}$. The annihilation operator $\oa$ and creation operator $\oa^\dag$ are unbounded operators that act as follows on this basis:
	\begin{align}
		\oa\ket{n+1}=\sqrt{n+1}\ket{n},\quad \oa^\dag \ket{n}=\sqrt{n+1}\ket{n+1}.
	\end{align}
    \item For a single bosonic mode, $\xH^k$ and $\xW^{k,1}$ are defined in \cref{def_Hk,def_Wk}. $\xH^f\subset \cap_{k \in \xN} \xH^k$ is the (finite) span of $(\ket{k})_{k \in \xN}$. It is a dense subset of all the $\xH^k$ equipped with their norm. Similarly, $\xW^f$ is the set of operators with support on $\xH^f$.
    \item For a single bosonic mode, an operator $\oA$ on $\xH$ is a polynomial in creation and annihilation operators if there exists a non-commutative polynomial $P\in \C[X,Y]$ such that it can be written\footnote{To avoid domain issues, we implicitly require that both operators are defined on the dense subset $\xH^f$.} as $P(\oa,\oa^\dag)$. The degree of the operator is then the smallest degree of such a polynomial $P$. 
    \item Let $\oP\in B(\xH)$ be an orthogonal projector; we denote by $\oP^\perp$ the associated projector onto the orthogonal complement of the range of $\oP$.
\end{itemize}
\vspace{1em}